\documentclass[11pt,a4paper,onesided]{article}

\usepackage{amssymb, epsfig}
\usepackage[a4paper, top=1in, bottom=1in, left=1.02in, right=1.02in ]{geometry}
\usepackage{latexsym}
\usepackage{setspace}
\usepackage{eepic}
\usepackage{amsmath}
\usepackage{amsthm}
\usepackage{makeidx}
\usepackage{graphicx}
\usepackage{stmaryrd}
\usepackage{bbding}
\usepackage{float}
\usepackage[latin1]{inputenc}  
\usepackage[english]{babel}
\usepackage{amssymb}
\usepackage{stmaryrd}
\usepackage{fancybox} 
\usepackage{fancyhdr}
\usepackage[usenames,dvipsnames]{color}
\usepackage{pdflscape}
\usepackage{pinlabel}
 \usepackage{cite}
\usepackage{hyperref}

\newcommand{\cc}{\chi}

\hypersetup{
bookmarks=true,         
    unicode=false,          
    pdftoolbar=true,        
    pdfmenubar=true,        
    pdfstartview={FitB},    
    pdftitle={A Lie-Rinehart algebra without an antipode on its enveloping algebra},    
    pdfauthor={Ulrich Kr\"ahmer, Ana Rovi},     
    pdfsubject={A Lie-Rinehart algebra without an antipode},   
    pdfcreator={Creator},   
    pdfproducer={Producer}, 
    pdfkeywords={Lie-Rinehart algebras}  {Antipodes} {Hopf algebroids}, 
    pdfnewwindow=true,      
    colorlinks=true,       
    linkcolor=Black,          
    citecolor=Black,        
    filecolor=Black,      
    urlcolor=Black           
}

\title{A Lie-Rinehart algebra with no antipode}

 \author{Ulrich Kr\"ahmer\footnote{ulrich.kraehmer@glasgow.ac.uk}  \ and Ana Rovi\footnote{a.rovi.1@research.gla.ac.uk} \\
School of Mathematics and Statistics\\ 
University of Glasgow\\  
Glasgow G12 8QW, UK}

\numberwithin{equation}{section}

\newtheorem{theorem}{Theorem}[section]

\newtheorem{proposition}[theorem]{Proposition}
\newtheorem{definition}[theorem]{Definition}

\begin{document}

\setcounter{secnumdepth}{3}
\setcounter{tocdepth}{3}

\maketitle

\begin{abstract}

The aim of this note is to communicate a simple example of a Lie-Rinehart algebra whose enveloping algebra is not a Hopf algebroid in the sense of B\"ohm and Szlach\'anyi.

\end{abstract}

\singlespacing


\section{Introduction}

The enveloping algebra of a Lie algebra is a classical 
example of a  Hopf algebra. Hence it is 
natural to ask whether the enveloping algebra of a Lie algebroid 
\cite{Pradines}
or more generally of a Lie-Rinehart algebra \cite{RinehartForms} 
carries the  structure of a Hopf algebroid.  It turns out that they  
always are  \emph{left bialgebroids} (introduced under the name  
$\times_R$-bialgebras by Takeuchi \cite{takeuchi}),  
see \cite{xu}, and in fact \emph{left Hopf algebroids} 
(introduced under the name $\times_R$-Hopf algebras by 
Schauenburg \cite{Schauenburg}), see  \cite[Example 2]{KOKR}; 
see also \cite{huebschmannhopf,iekeenveloping}. 

However, the question whether these left Hopf algebroids are  
\emph{full Hopf  algebroids} in the sense of \cite{Gabriella,JHLu} 
remained open. In the light of \cite[Proposition~3.11]{NielsPosthuma}, 
this is known to be true for Lie algebroids \cite{WeinsteinTransverse} and 
for the Lie-Rinehart algebras associated to Poisson
algebras   \cite[Section (3.2)]{HuebschmannPaper1}. A
counterexample was announced by Kowalzig and the first
author, see  \cite[Remark~3.12]{NielsPosthuma}, but the
construction contained a gap.    To our knowledge, the
literature still contains no   example of a left Hopf
algebroid that is not a full one. Hence the aim of the
present note is to communicate such an example:

\begin{theorem}\label{main}
Let $K$ be a field, $R:= K  [ x, y  ] / \langle x
\cdot y , x^2 , y^2 \rangle$,  $L$ be the 1-dimensional
Lie algebra with basis  $ \{\alpha \}$ and
$E \in \mathrm{Der}_K(R)$ be the derivation  with $E(x)=y,E(y)=0$. 
\begin{enumerate}
\item There is a Lie-Rinehart algebra structure on  $(R,L)$ with $R$-module structure on $L $ given by $x \cdot \alpha= y \cdot \alpha = 0$    and  anchor map given by $ \rho (\alpha)=E$.
\item There is no right $V(R,L)$-module structure on $R$ that extends right multiplication in $R$. In  particular, $V(R,L)$ is not a full Hopf algebroid. 
\end{enumerate}
\end{theorem}

The note is structured as follows: in Section \ref{section: background} we recall some basic definitions. In Section \ref{MainThm}  we  provide a construction method of Lie-Rinehart algebras whose enveloping algebras do not admit an antipode. The simplest example of these is the one in  our theorem.   Lastly, Section \ref{antipode} illustrates the result  by giving an explicit presentation of $V(R,L)$ for our example in which the  nonexistence of an antipode becomes evident.\\   

U.K. acknowledges support by the EPSRC grant ``Hopf algebroids and Operads''; 
A.R. is funded by an EPSRC DTA grant and thanks Jos\'e Figueroa O'Farrill for 
his encouragement, and Gwyn Bellamy for remarks about differential  operators.

\section{Background}
\label{section: background}

This section contains background on Lie-Rinehart  algebras \cite{RinehartForms}, 
see also
\cite{HuebschmannPaper1,NielsThesis,NielsPosthuma,iekeenveloping} for more information. For the corresponding differential geometric notion of a Lie algebroid see  \cite{Pradines} and for example \cite{Mackenzie} for further details.

We fix a field $K$. An unadorned $ \otimes $
denotes the tensor product of $K$-vector spaces.
\begin{definition}
\label{def LR}

A Lie-Rinehart algebra consists of  
\begin{enumerate} 
\item a commutative $K$-algebra $(R, \cdot )$,
\item a Lie algebra $( L , [-,-]_L ) $ over  $K$,
\item a left $R$-module structure  $R \otimes L \rightarrow L$, $r \otimes \xi \mapsto  	r \cdot \xi$, $r \in R, \xi \in L$, and
\item an  $R$-linear Lie algebra homomorphism  $ \rho: L \rightarrow \mathrm{Der}_K ( R )$ satisfying  
\begin{equation}
  \label{condition1}
	[ \xi, r \cdot \zeta ]_L = r \cdot [ \xi , \zeta ]_L +  	\rho ( \xi ) ( r ) \cdot \zeta,  	\quad r \in R, \xi, \zeta \in L.
\end{equation}
\end{enumerate}        %
The map $ \rho $ is referred to as the anchor map.
\end{definition}

There are two fundamental examples: if $R$ is any
commutative algebra, one can take $L$ to be 
$\mathrm{Der}_K(R)$ with its usual Lie algebra and
$R$-module structure, and $ \rho =  \mathrm{id} $. 
The other extreme is $R=K$ and $ \rho = 0$, $L$ being
any Lie algebra. 

In his paper \cite{RinehartForms}, Rinehart generalised   
the construction of the universal enveloping algebra of a Lie algebra to Lie-Rinehart algebras, see Section~2   therein for the precise construction.   The result is an associative $K$-algebra $ V(R,L)$ that is generated by the (sum of the) images of a $K$-algebra map  
$$
	R \longrightarrow V ( R, L )
$$
and a Lie algebra map 
$$
	( L, [-,-]_L ) \longrightarrow ( V ( R,L ), [-,-]
),\quad
	\xi \longmapsto \bar \xi
$$
where $[-,-]$ denotes the commutator in $ V ( R,L )$.  
As Rinehart, we do not distinguish between an element in 
$R$ and its image in $V(R,L)$ which is justified as the
first map is always injective. The construction 
is such  that in $V(R,L)$ one has for all 
$r \in R,\xi \in L$
\begin{equation}\label{xr}
	[\bar\xi,r]=\rho(\xi)(r),\quad r \bar\xi
=\overline{r \cdot \xi},
\end{equation}
where the product in $V(R,L)$ is denoted by concatenation.

As indicated in the introduction, $V(R,L)$ has the
structure of a left Hopf algebroid. Its counit endows
$R$ with the structure of a left $V(R,L)$-module, in
such a way that the induced action of $r \in R$ is
given by  left multiplication, and the induced action
of $ \xi \in L$ is given by the anchor map. For a full
Hopf  algebroid, composing the counit with the antipode
yields also a right $V(R,L)$-module structure on the
base algebra $R$ extending right multiplication in $R$,
see \cite[Proposition~3.11]{NielsPosthuma} for full
details. Thus the nonexistence of such a right module
structure on the base algebra $R$ indeed implies the nonexistence of an antipode.

\section{Proof of Theorem~\ref{main}}
\label{MainThm}

We now prove Theorem~\ref{main}. We begin by  considering more generally    Lie-Rinehart algebras $(R,L)$   whose $R$-module structure on $L$ is given by a  
character $ \cc : R \rightarrow K$.  
\begin{proposition}\label{final}

Let $(R,\cdot)$ be a commutative  $K$-algebra, $(L,[-,-]_L)$ be a Lie algebra and $ \rho : L  \rightarrow \mathrm{Der}_K(R)$ be a Lie algebra map. Define an $R$-module structure on $L$ by  $r \cdot \xi := \cc(r)\xi$, where $\cc: R \rightarrow K$ is a character on $R$. Then $(R,L)$ is a Lie-Rinehart algebra if and only if $ \rho $ is $R$-linear and $\rho(\xi)(r) \in \mathrm{ker}\, \cc$ for all $r \in R,\xi \in L$.
\end{proposition}

\begin{proof}
This follows as the Leibniz rule (\ref{condition1}) takes the form
$$
	[\xi,\cc(r)\zeta]_L = 	\cc(r)[\xi,\zeta]_L+\cc(\rho(\xi)(r))	\zeta
$$
and hence by the $K$-linearity of the bracket becomes  equivalent to $ \rho(\xi)(r) \in \mathrm{ker}\cc$.
\end{proof}

Note that for these examples, $[-.-]_L$ is even
$R$-linear, so $L$ is a Lie algebra over $R$. However,
in general we have $ \rho \neq 0$.  

Assume now that $(R,L)$ is a Lie-Rinehart algebra as in the above 
proposition, and that right multiplication in $R$ can be  extended to a 
right $V(R,L)$-module structure on $R$.  Denote by 
$\partial(\xi) \in R$ the element obtained by acting with 
$\xi \in L$ on $1 \in R$ under this right action.  This defines 
a $K$-linear map $\partial : L  \rightarrow R$, and  
in $V(R,L)$ we have 
$$
	\rho(\xi)(r)=[\bar\xi,r]=\bar\xi r-r \bar\xi=\bar\xi
r - \overline {r \cdot \xi}= 	\bar\xi r -\cc(r)\bar\xi,
$$  
so by acting with this element on $1 \in R$, one sees that 
this map $\partial$ satisfies
\begin{equation}\label{huebsch}
	\rho (\xi)(r)=\partial(\xi) \cdot (r - \cc(r)).
\end{equation}
A $K$-linear map $ \partial $ with  this property defines a right 
$V(R,L)$-module structure  extending multiplication on $R$ if and only if 
it  satisfies the condition 
$\partial([\xi,\zeta]_L)=\rho(\xi)(\partial(\zeta))-\rho(\zeta)
(\partial(\xi))$. It also corresponds to a generator of the Gerstenhaber bracket on $\Lambda_R L$, see \cite{HuebschmannPaper1},  but we shall not need these facts:
\begin{proof}[Proof of Theorem~\ref{main}]

The first part is verified by explicit computation;  the Lie-Rinehart algebra is of the form as in Proposition~\ref{final} with $\cc$ given by  $\cc(x)=\cc(y)=0$.  

 For 2., take $r=x$ and $\xi=\alpha$ in (\ref{huebsch}). One obtains $y=E(x)=\rho(\alpha)(x)=\partial(\alpha)x$. However,  there is no element $z \in R$ such that $y=z \cdot x$.
\end{proof}

\section{A Hopf algebroid without antipode}
\label{antipode}
Carrying out Rinehart's construction explicitly yields 
a  presentation of the associative $K$-algebra $V(R,L)$ in terms of  
generators $x,y,\bar\alpha$ satisfying the relations 
$$
	\bar\alpha x=y,\quad
	\bar\alpha y=x \bar\alpha = y \bar\alpha = x^2=y^2=xy=yx=0.
$$
Hence $V(R,L)$ has a $K$-linear basis given by 
$\{\bar\alpha^n,x,y\}_{n \in \mathbb{N}}$.
   
In view of Axiom (iii) in
\cite[Definition~4.1]{Gabriella}, the antipode $S$ of
any Hopf algebroid $H$ over $R$ 
satisfies $S(t(r))=s(r)$  where $s,t : R \rightarrow
V(R,L)$ are the source and  
the target map of the underlying left bialgebroid, respectively. 
For the left bialgebroid $V(R,L)$, these are both the 
inclusion of $R$ into $V(R,L)$, 
hence an antipode on $V(R,L)$ would satsify  
$S(x)=x$, $S(y)=y$.

However, the antipode of a Hopf algebroid is an algebra antihomomorphism, 
$S(gh)=S(h)S(g)$ for all $g,h \in H$, see 
e.g.~\cite[Proposition~4.4.1]{Gabriella}. So in
$V(R,L)$,
one would have 
$$
	y=S(y)=S(\bar\alpha x)=S(x)S(\bar\alpha)=x
S(\bar\alpha).
$$
This illustrates directly that $V(R,L)$ admits no
antipode, since there is
no element $z \in V(R,L)$ such that $y=xz$.

 \renewcommand{\bibname}{References}
\bibliographystyle{alpha}

\begin{thebibliography}{ELW99}

\bibitem[B{\"o}h09]{Gabriella}
G.~B{\"o}hm.
\newblock Hopf algebroids.
\newblock In {\em Handbook of algebra. {V}ol. 6}, volume~6 of {\em Handb.
  Algebr.}, pages 173--235. Elsevier/North-Holland, Amsterdam, 2009.

\bibitem[ELW99]{WeinsteinTransverse}
S.~Evens, J.-H. Lu, and A.~Weinstein.
\newblock Transverse measures, the modular class and a cohomology pairing for
  {L}ie algebroids.
\newblock {\em Quart. J. Math. Oxford Ser. (2)}, 50(200):417--436, 1999.

\bibitem[Hue98]{HuebschmannPaper1}
J.~Huebschmann.
\newblock Lie-{R}inehart algebras, {G}erstenhaber algebras and
  {B}atalin-{V}ilkovisky algebras.
\newblock {\em Ann. Inst. Fourier (Grenoble)}, 48(2):425--440, 1998.

\bibitem[Hue08]{huebschmannhopf}
J.~Huebschmann.
\newblock The universal {H}opf algebra associated with a
  {H}opf-{L}ie-{R}inehart algebra.
\newblock preprint arXiv:0802.3836, February 2008.

\bibitem[KK10]{KOKR}
N.~Kowalzig and U.~Kr{\"a}hmer.
\newblock Duality and products in algebraic (co)homology theories.
\newblock {\em J. Algebra}, 323(7):2063--2081, 2010.

\bibitem[Kow09]{NielsThesis}
N.~Kowalzig.
\newblock {\em Hopf algebroids and their cyclic theory}.
\newblock PhD thesis, Universiteit Utrecht, 2009.

\bibitem[KP11]{NielsPosthuma}
N.~Kowalzig and H.~Posthuma.
\newblock The cyclic theory of {H}opf algebroids.
\newblock {\em J. Noncommut. Geom.}, 5(3):423--476, 2011.

\bibitem[Lu96]{JHLu}
J.-H. Lu.
\newblock Hopf algebroids and quantum groupoids.
\newblock {\em Internat. J. Math.}, 7(1):47--70, 1996.

\bibitem[Mac87]{Mackenzie}
K.~Mackenzie.
\newblock {\em Lie groupoids and {L}ie algebroids in differential geometry},
  volume 124 of {\em London Mathematical Society Lecture Note Series}.
\newblock Cambridge University Press, Cambridge, 1987.

\bibitem[MM10]{iekeenveloping}
I.~Moerdijk and J.~Mr{\v{c}}un.
\newblock On the universal enveloping algebra of a {L}ie algebroid.
\newblock {\em Proc. Amer. Math. Soc.}, 138(9):3135--3145, 2010.

\bibitem[Pra67]{Pradines}
J.~Pradines.
\newblock Th\'eorie de {L}ie pour les groupo\"\i des diff\'erentiables.
  {C}alcul diff\'erenetiel dans la cat\'egorie des groupo\"\i des
  infinit\'esimaux.
\newblock {\em C. R. Acad. Sci. Paris S\'er. A-B}, 264:A245--A248, 1967.

\bibitem[Rin63]{RinehartForms}
G.~S. Rinehart.
\newblock Differential forms on general commutative algebras.
\newblock {\em Trans. Amer. Math. Soc.}, 108:195--222, 1963.

\bibitem[Sch00]{Schauenburg}
P.~Schauenburg.
\newblock Duals and doubles of quantum groupoids ({$\times_R$}-{H}opf
  algebras).
\newblock In {\em New trends in {H}opf algebra theory ({L}a {F}alda, 1999)},
  volume 267 of {\em Contemp. Math.}, pages 273--299. Amer. Math. Soc.,
  Providence, RI, 2000.

\bibitem[Tak77]{takeuchi}
M.~Takeuchi.
\newblock Groups of algebras over {$A\otimes \overline A$}.
\newblock {\em J. Math. Soc. Japan}, 29(3):459--492, 1977.

\bibitem[Xu01]{xu}
P.~Xu.
\newblock Quantum groupoids.
\newblock {\em Comm. Math. Phys.}, 216(3):539--581, 2001.

\end{thebibliography}

\end{document}